\documentclass[12pt,english]{amsart}
\usepackage{graphicx}
\usepackage{color}
\usepackage{pdfsync,cite}
\allowdisplaybreaks
\usepackage[T1]{fontenc}
\usepackage{mathrsfs}
\usepackage[utf8]{inputenc}

\newtheorem{thm}{Theorem}[section]

\newtheorem{lemma}[thm]{Lemma}

\theoremstyle{definition}
\newtheorem{definition}[thm]{Definition}

\theoremstyle{remark}

\numberwithin{equation}{section}

\newcommand*{\diff}{\mathop{}\mathrm{d}}

\newcommand{\eqnum}{\refstepcounter{equation}\textup{\tagform@{\theequation}}}

\begin{document}

\title[]
{Martingale solution to stochastic extended Korteweg - de Vries equation}

\author[Karczewska]{Anna Karczewska}
\address{Faculty of Mathematics, Computer Science and Econometrics, University of Zielona G\'ora, Szafrana 4a, 65-516 Zielona G\'ora, Poland} \email{A.Karczewska@wmie.uz.zgora.pl} \thanks{}

\author[Szczeci\'nski]{Maciej Szczeci\'nski} \address{ Faculty of Mathematics, Computer Science and Econometrics, University of Zielona G\'ora, Szafrana 4a, 65-516 Zielona G\'ora, Poland} \email{mszczecinski@wmie.uz.zgora.pl} %\thanks{The author is partially supported by  FONDECYT grant number 1140258}

\date{\today}



%%% ----------------------------------------------------------------------

\begin{abstract}
We study a stochastic extended Korteweg - de Vries equation driven by a multiplicative noise. We prove the existence of a martingale solution to the equation studied. The proof of the solution is based on two approximations of the problem considered and the compactness method.
 \end{abstract}

%%% ----------------------------------------------------------------------
\maketitle
%%% ----------------------------------------------------------------------

\section{Introduction} \label{s1}

The celebrated Korteweg - de Vries equation (\emph{KdV} for short) \cite{kdv}, derived from the set of Eulerian shallow water and long wavelength equations, become a paradigm in the field of nonlinear partial differential equations. 
 KdV appears as the lowest approximations of wave motion in several fields of physics, see, e.g., monographs \cite{DrJ,EIGR,Newell,Rem,Whit} and references therein.

 KdV is, however, the result of an approximation of the set of the Euler equations  within perturbation approach limited to the first order in expansion with respect to parameters assumed to be small. 
Several authors have extended KdV to the second order (the \emph{extended KdV} or \emph{KdV2}),  e.g. \cite{MS90,MS96,BS,KRI,KRR,KRI2,KRIR,Yang},
which is a more exact approximation of the Euler equations but far more difficult since it contains higher nonlinear term and higher derivatives.

Despite its non-integrability, KdV2 has exact analytic solutions both solitonic \cite{KRI} and periodic \cite{RK,RKI}. 
These solutions have the same form as corresponding solutions to KdV but with slightly different coefficients.

A natural continuation of the study of the extended KdV equation seems to be considering stochastic versions of such equation. KdV2 equation driven by random noise can be a model of several kinds of waves (e.g., surface water waves, waves in plasma) influenced by random factors. Two cases of the stochastic KdV2 equation are possible - the case with additive noise and the case with the multiplicative noise. The additive case we  studied in \cite{KaSz1}, where a mild solution to KdV2 has been established.

In this paper, we  consider the stochastic extended Korteweg - de Vries equation with  multiplicative random noise. We prove the existence of  martingale solution to stochastic KdV2 equation driven by cylindrical Wiener process. In the proof, we generalize the methods used in papers  \cite{Deb} and \cite{Gat}. We have to emphasize that the method used in \cite{Deb} for estimations was not suitable in our case. We adapted for our purposes (proof of Lemma~2.4) the approach used by Flandoli and Gątarek in \cite{Gat}. 

%%%%%%%%%%%%%%%%%%%%%%%%%%%%%%%%%%%%%%%%%%%%%%
\section{Existence of martingale solution}

%Let  $\left( \Omega, \mathscr{F}, \left\{\mathscr{F}_{t}\right\}_{t\geq 0}, \mathbb{P} \right)$ be a probability space with filtration.
We consider initial value problem for Korteweg - de Vries type equation
\begin{equation}\label{Apr}
\begin{cases}
\diff u(t,x) + \big[ u_{3x}(t,x) + u(t,x)u_{x}(t,x) + u(t,x)u_{3x}(t,x)  \\
 \hspace{11ex} +\, 3 u_{x}(t,x)u_{2x}(t,x) \big] \diff t = \Phi\left(u(t,x)\right) \diff W(t), \\
u(0,x) = u_{0}(x).
\end{cases}
\end{equation}
In (\ref{Apr}), 
$W(t)$, $t\geq 0$, is a cylindrical Wiener process adapted to the filtration 
 $\left\{\mathscr{F}_{t}\right\}_{t\geq 0}$, \linebreak
$u_{0} \in L^{2}(\mathbb{R}) $ is a deterministic function, $u(\omega,\cdot,\cdot):\mathbb{R}_{+}\times\mathbb{R} \rightarrow \mathbb{R}$ for all $\omega\in\Omega$. Moreover, we assume that
 $|u(t,x)|+|u|_{L^{2}(\mathbb{R})}<\lambda<\infty, ~\lambda>0, \mbox{~for~all~} t\in\mathbb{R}_{+} \mbox{~and~} x\in\mathbb{R}$, what reflects finitness of solutions to deterministic version of the equation (\ref{Apr}) (see, e.g., \cite{KRI,RK,RKI}).
The operator $\Phi$  is a continuous mapping from
$H^{2}(\mathbb{R})$ to $L_{2}^{0}(L^{2}(\mathbb{R}))$, the space of Hilbert-Schmidt operators from $L^{2}(\mathbb{R})$ to itself. 
The operator  $\Phi$  is such that for any  $u\in H^{2}(\mathbb{R})$ 
the following conditions hold:
%\begin{itemize}\item[]	
\begin{equation}\label{W1}
\displaystyle\mathop{\exists}_{\kappa_{1}, \kappa_{2}>0} \quad \left\|\Phi(u(x))\right\|_{L_{0}^{2}(L^{2}(\mathbb{R}))} \leq \kappa_{1} \max \left\{\left|u(x)\right|^{2}_{L^{2}(\mathbb{R})}, \left|u(x)\right|_{L^{2}(\mathbb{R})}\right\} + \kappa_{2};
\end{equation}
%\item[]
\begin{align}\label{W3}
&\mbox{there~exist~such~functions}~  a,b\in L^{2}(\mathbb{R})~ \mbox{with~compact~support,~that~the~mapping} \\ 
&u \mapsto \left( \Phi(u)a , \Phi(u)b \right)_{L^{2}(\mathbb{R})}~ \mbox{is~continuous~in~topology}~ L^{2}_{loc}(\mathbb{R}).\nonumber
\end{align}
%\end{itemize}

\begin{definition} We say that the problem  (\ref{Apr}) has a {\tt  martingale solution} on the interval  $[0,T]$, $T>0$, if there exists a stochastic basis  $(\Omega,\mathscr{F},\left\{\mathscr{F}_{t}\right\}_{t\geq 0}, \mathbb{P}, \left\{W_{t}\right\}_{t\geq 0} )$, where $\left\{W_{t}\right\}_{t\geq 0}$ is a cylindrical Wiener process, and there exists the process  $\left\{u(t,x)\right\}_{t\geq 0}$ adapted to the filtration  $\left\{\mathscr{F}_{t}\right\}_{t\geq 0}$ with trajectories belonging to the space 
\begin{equation}\nonumber
%u(\omega,\cdot,\cdot) \in 
L^{\infty}(0,T;L^{2}(\mathbb{R}))\cap L^{2}(0,T;L^{2}_{loc}(\mathbb{R})\cap \mathscr{C}(0,T;H^{s}_{loc}(\mathbb{R}), \quad s<0, \quad \mathbb{P} - \text{a.s.}  
\end{equation}
such that 
\begin{equation}\nonumber
\begin{aligned}
& \left\langle u(t,x), v(x) \right\rangle + \int_{0}^{t} \left\langle u_{3x}(s,x) + u(s,x)u_{x}(s,x) + u(s,x)u_{3x}(s,x) \right. \\
& \left. + 3 u_{x}(s,x)u_{2x}(s,x), v(x) \right\rangle \diff s = \left\langle u_{0}(x), v(x)\right\rangle + \left\langle \int_{0}^{t} \Phi(u(s,x)) \diff W(s), v(x) \right\rangle
\end{aligned}
\end{equation}
for any $t\in[0,T]$ and $v\in H^{1}_{loc}(\mathbb{R})$.
\end{definition}

Now, we can to formulate the main result of the paper.

\begin{thm}\label{P4.1}
For all  $u_{0}\in L^{2}(\mathbb{R})$  and $T>0$ there exists a martingale solution to (\ref{Apr}) with conditions (\ref{W1}) and (\ref{W3}).% Moreover, $u$ has trajectories in $L^{\infty}(0,T;L^{2}(\mathbb{R}))\cap L^{2}(0,T;L^{2}_{loc}(\mathbb{R}))\cap \mathscr{C}(0,T;H^{s}_{loc}(\mathbb{R}))$ for all $s<0$.
\end{thm}
\begin{proof}

Let $\varepsilon>0$. Consider

\begin{equation}\label{par}
\begin{cases}
\diff u^{\varepsilon}(t,x) + \big[ \varepsilon u^{\varepsilon}_{4x}(t,x) + u^{\varepsilon}_{3x}(t,x) +  u^{\varepsilon}(t,x)u^{\varepsilon}_{x}(t,x) + 3u_{x}^{\varepsilon}(t,x)u^{\varepsilon}_{2x}(t,x) \\ 
\hspace{12ex} + u^{\varepsilon}(t,x)u^{\varepsilon}_{3x}(t,x) \big] \diff t 
 =  \Phi\left(u^{\varepsilon}(t,x)\right)\diff W(t) \\
u^{\varepsilon}_{0}(x) = u^{\varepsilon}(0,x) .
\end{cases}
\end{equation}

\begin{lemma}\label{parMart}
For any $\varepsilon>0$ there exists a martingale solution to the problem  (\ref{par}) with conditions (\ref{W1}) and (\ref{W3}).
\end{lemma}

\pagebreak

\begin{lemma}\label{szac4.1}
There exists $\varepsilon_{0} > 0$, such that 
\begin{eqnarray}
\label{4.1a}
\exists_{C_{1}>0}\forall_{0 < \varepsilon < \varepsilon_{0}}  \varepsilon\mathbb{E}\left( \left|u^{\varepsilon}(t,x)\right|^{2}_{L^{2}(0,T;H^{2}(\mathbb{R}))} \right)  & \leq \tilde{C}_{1},  \\ 
\label{4.1c} 
\forall_{k\in X_{k}}\exists_{C_{2}(k)>0}\forall_{0 < \varepsilon < \varepsilon_{0}} \mathbb{E}\left( \left|u^{\varepsilon}(t,x)\right|^{2}_{L^{2}(0,T;H^{1}(-k,k))} \right)  & \leq \tilde{C}_{2}(k), 
\end{eqnarray}
where $X_{k} = \big\{k>0: \left|k\right| \leq \min\left\{-x_{1}, x_{2}\right\}\big\}$.
\end{lemma}

\begin{lemma}\label{PropCias}
The family of distributions  $\mathscr{L}(u^{\varepsilon})$ is tight in $L^{2}(0,T;L^{2}_{loc})\cap\mathscr{C}(0,T;H^{-3}_{loc}(\mathbb{R}))$.
\end{lemma}

Proofs of Lemmas \ref{parMart}, \ref{szac4.1} and  \ref{PropCias} are given in sections \ref{sec3} and \ref{sec4}.

Substitute in Prohorov's theorem (e.g., see Theorem~5.1 in \cite{Bil})
  $S:=L^{2}(0,T;L^{2}_{loc}(\mathbb{R}))\cap\mathscr{C}(0,T;H^{-3}_{loc}(\mathbb{R}))$ and $\mathscr{K}:=\left\{\mathscr{L}(u^{\varepsilon})\right\}_{\varepsilon>0}$. 
Since $\mathscr{K}\subset\mathscr{P}(S)$ is tight in  $S$, 
then it is  sequentially compact, so there exists a subsequence of $\left\{\mathscr{L}(u^{\varepsilon})\right\}_{\varepsilon>0}$ converging to some measure $\mu$ in~$\bar{\mathscr{K}}$. 

Because $\left\{\mathscr{L}(u^{\varepsilon})\right\}_{\varepsilon>0}$ is convergent, then it is also weakly convergent. Therefore in Skorohod's theorem (e.g., see Theorem~6.7 in \cite{Bil})
 one can substitute $\mu_{\varepsilon}:=\left\{\mathscr{L}(u^{\varepsilon})\right\}_{\varepsilon>0}$, \linebreak $\mu:=\lim_{\varepsilon\to 0} \mu_{\varepsilon}$.
Then there exists a space
 $(\bar{\Omega}, \bar{\mathscr{F}}, \left\{\bar{\mathscr{F}}_{t}\right\}_{t\geq 0}, \bar{\mathbb{P}})$ and random variables $\bar{u}^{\varepsilon}$, $\bar{u}$ with values in $L^{2}(0,T;L^{2}_{loc})\cap \mathscr{C}(0,T;H^{-3}_{loc}(\mathbb{R}))$ such that $\bar{u}^{\varepsilon}\rightarrow\bar{u}$ in $L^{2}(0,T;L^{2}_{loc})$ and in $\mathscr{C}(0,T;H^{-3}_{loc}(\mathbb{R}))$. Moreover  $\mathscr{L}(\bar{u}^{\varepsilon}) = \mathscr{L}(u^{\varepsilon})$. 

Then due to Lemma  \ref{szac4.1}, for any $p\in\mathbb{N}$ there exist constants $\tilde{C}_{1}(p)$, $\tilde{C}_{2}$ such that
\begin{equation}\nonumber
\begin{aligned}
\mathbb{E}(\sup_{t\in[0,T]} \left|\bar{u}^{\varepsilon}(t,x)\right|_{L^{2}(\mathbb{R})}^{2p}) \leq \tilde{C}_{1}(p) \quad \mbox{and} \quad
\mathbb{E}(\left|\bar{u}^{\varepsilon}(t,x)\right|^{2}_{L^{2}(0,T;H^{2}(\mathbb{R}))}) \leq \tilde{C}_{2}.
\end{aligned}
\end{equation}
Additionally, 
\begin{equation}\nonumber
\bar{u}^{\varepsilon}(t,x)\in L^{2}(0,T; H^{1}(-k,k))\cap L^{\infty}(0,T; L^{2}(\mathbb{R})), \quad \mathbb{P} - \text{a.s.} 
\end{equation}
Then one can conclude that
 $\bar{u}^{\varepsilon}\rightarrow \bar{u}$ weakly in $L^{2}(\bar{\Omega},L^{2}(0,T;H^{1}(-k,k)))$. 

Let $x\in\mathbb{R}$ be fixed and denote
\begin{equation}\nonumber
\begin{aligned}
M^{ \varepsilon}(t) := & u^{ \varepsilon}(t,x) - u_{0}^{  \varepsilon}(x) + \int_{0}^{t}\bigg[ \varepsilon u^{\varepsilon}(t,x)_{4x}(t,x) + u^{ \varepsilon}(t,x)u^{ \varepsilon}_{x}(t,x)\\
& + u^{ \varepsilon}_{3x}(t,x) + 3u^{ \varepsilon}_{x}(t,x) u^{ \varepsilon}_{2x}(t,x) + u^{ \varepsilon}(t,x) u^{ \varepsilon}_{3x}(t,x) \bigg] \diff s,  \\
\bar{M}^{ \varepsilon}(t) := & \bar{u}^{ \varepsilon}(t,x) - \bar{u}_{0}^{  \varepsilon}(x) + \int_{0}^{t}\bigg[ \varepsilon \bar{u}^{\varepsilon}(t,x)_{4x}(t,x) + \bar{u}^{ \varepsilon}(t,x)\bar{u}^{ \varepsilon}_{x}(t,x)\\
& + \bar{u}^{ \varepsilon}_{3x}(t,x) + 3\bar{u}^{ \varepsilon}_{x}(t,x) \bar{u}^{ \varepsilon}_{2x}(t,x) + \bar{u}^{ \varepsilon}(t,x) \bar{u}^{ \varepsilon}_{3x}(t,x) \bigg] \diff s.  
\end{aligned}
\end{equation}

Note, that 
\begin{equation}\nonumber
\begin{aligned}
M^{ \varepsilon}(t) = & u_{0}^{  \varepsilon}(x) - \int_{0}^{t}\bigg[ \varepsilon u^{\varepsilon}(t,x)_{4x}(t,x) + u^{ \varepsilon}(t,x)u^{ \varepsilon}_{x}(t,x) + u^{ \varepsilon}_{3x}(t,x) + 3u^{ \varepsilon}_{x}(t,x) u^{ \varepsilon}_{2x}(t,x)\\
& + u^{ \varepsilon}(t,x) u^{ \varepsilon}_{3x}(t,x) \bigg] \diff s + \int_{0}^{t} \left( \Phi \left( u^{  \varepsilon} (s,x) \right) \right) \diff W^{  \varepsilon} (s) - u_{0}^{  \varepsilon}(x) +  \int_{0}^{t}\bigg[ \varepsilon u^{\varepsilon}(t,x)_{4x}(t,x) \\
& + u^{ \varepsilon}(t,x)u^{ \varepsilon}_{x}(t,x) + u^{ \varepsilon}_{3x}(t,x) + 3u^{ \varepsilon}_{x}(t,x) u^{ \varepsilon}_{2x}(t,x) + u^{ \varepsilon}(t,x) u^{ \varepsilon}_{3x}(t,x) \bigg] \diff s  \\
= & \int_{0}^{t} \left( \Phi \left( u^{  \varepsilon} (s,x) \right) \right) \diff W (s), \\
\end{aligned}
\end{equation}
so, $M^{ \varepsilon}(t)$, $t\geq 0$, is a square integrable martingale with values in  $L^{2}(\mathbb{R})$, adapted to filtration $\sigma\left\{u^{  \varepsilon}(s,x), 0\leq s \leq t\right\}$ with quadratic variation 
\begin{equation}\nonumber
\left[M^{ \varepsilon}\right](t) = \int_{0}^{t}\Phi(u^{  \varepsilon}(s,x))\left[\Phi(u^{  \varepsilon}(s,x))\right]^{*} \diff s .
\end{equation}

Substitute in the Doob inequality (e.g., see Theorem 2.2 in \cite{Gaw})
 $M_{t}:=M^{  \varepsilon}(t)$ and $p:=2p$. Then
\begin{equation}\label{szacDoob}
\mathbb{E}\left[\left( \sup_{t\in[0,T]} \left|M^{  \varepsilon}(t)\right|_{L^{2}(\mathbb{R})}^{p} \right)\right] \leq \left( \frac{p}{p-1} \right)^{p} \mathbb{E} \left(\left|M^{  \varepsilon}(T)\right|_{L^{2}(\mathbb{R})} \right).
\end{equation}

Assume $0\leq s \leq t \leq T$ and let $\varphi$ be a bounded continuous function on $L^{2}(0,s;L^{2}_{loc}(\mathbb{R}))$ or $C(0,s;H^{-3}_{loc}(\mathbb{R}))$. Let $a,b\in H^{3}_{0}(-k,k)$, $k\in\mathbb{N}$, be arbitrary and fixed. Since $M^{ \varepsilon}(t)$ is a martingale and $\mathscr{L}(\bar{u}^{ \varepsilon}) = \mathscr{L}(u^{ \varepsilon})$, then (see \cite{Gat}, p. 377-378)
\begin{equation}\nonumber
\begin{aligned}
\mathbb{E} & \Big( \left\langle M^{ \varepsilon}(t) - M^{ \varepsilon}(s); a \right\rangle \varphi\left(u^{ \varepsilon}(t,x)\right)\Big) = 0 , \\
\mathbb{E} & \Big( \left\langle \bar{M}^{ \varepsilon}(t) - \bar{M}^{ \varepsilon}(s); a \right\rangle\varphi\left(\bar{u}^{ \varepsilon}(t,x)\right)\Big) = 0 \\
\end{aligned}
\end{equation}
and
\begin{equation}\nonumber
\begin{aligned}
\mathbb{E} & \bigg\{\bigg[\left\langle  M^{ \varepsilon}(t);a \right\rangle \left\langle M^{ \varepsilon}(t);b \right\rangle - \left\langle M^{ \varepsilon}(s);a \right\rangle \left\langle M^{ \varepsilon}(s);b \right\rangle \\
& - \int_{s}^{t} \left\langle \left[\Phi\left(u^{ \varepsilon}(\xi,x)\right)\right]^{*}a ; \left[\Phi\left(u^{ \varepsilon}(\xi,x)\right)\right]^{*}b \right\rangle \diff \xi\bigg]\varphi(u^{ \varepsilon}(t,x))\bigg\} = 0, \\
\mathbb{E} & \bigg\{\bigg[\left\langle  \bar{M}^{ \varepsilon}(t);a \right\rangle \left\langle \bar{M}^{ \varepsilon}(t);b \right\rangle - \left\langle \bar{M}^{ \varepsilon}(s);a \right\rangle \left\langle \bar{M}^{ \varepsilon}(s);b \right\rangle \\
& - \int_{s}^{t} \left\langle \left[\Phi\left(\bar{u}^{ \varepsilon}(\xi,x)\right)\right]^{*}a ; \left[\Phi\left(\bar{u}^{ \varepsilon}(\xi,x)\right)\right]^{*}b \right\rangle \diff \xi\bigg]\varphi(\bar{u}^{ \varepsilon}(t,x))\bigg\} = 0 .
\end{aligned}
\end{equation}

Denote $$\bar{M}(t) := \bar{u}(t,x) - \bar{u}_{0}^{  \varepsilon}(x) + \int_{0}^{t}\bigg[ \bar{u}(t,x)\bar{u}_{x}(t,x)
 + \bar{u}_{3x}(t,x) + 3\bar{u}_{x}(t,x) \bar{u}_{2x}(t,x) + \bar{u}(t,x) \bar{u}_{3x}(t,x) \bigg] \diff s.$$ 
If $\varepsilon \rightarrow 0 $, to
$ \bar{M}^{ \varepsilon}(t) \rightarrow \bar{M}(t)$ and $\bar{M}^{  \varepsilon}(s) \rightarrow \bar{M}(s)$, $\bar{\mathbb{P}} - \text{a.s.}$ in $H^{-3}_{loc}(\mathbb{R})$. Moreover, since $\varphi$ is continuous, then $\varphi(\bar{u}^{  \varepsilon}(s,x)) \rightarrow \varphi(\bar{u}(s,x))$, $\bar{\mathbb{P}} - \text{a.s.}$. Therfeore, if $\varepsilon \rightarrow 0$, then
\begin{equation}\nonumber
\begin{aligned}
\mathbb{E}&  \Big( \left\langle \bar{M}^{  \varepsilon}(t) - \bar{M}^{  \varepsilon}(s); a \right\rangle \varphi(\bar{u}^{  \varepsilon}(t,x))\Big)  \rightarrow \mathbb{E} \Big( \left\langle \bar{M}(t) - \bar{M}(s); a \right\rangle \varphi(\bar{u}(t,x))\Big) .
\end{aligned}
\end{equation}

Additionaly, because $\Phi$ is a continuous operator in topology $L_{loc}^{2}(\mathbb{R})$ and (\ref{szacDoob}) holds, therefore if  $\varepsilon \rightarrow 0$, then
\begin{equation}\nonumber
\begin{aligned}
\left\langle \left(\Phi(\bar{u}^{ \varepsilon}(s,x))\right)^{*}a; \left(\Phi(\bar{u}^{ \varepsilon}(s,x))\right)^{*}b\right\rangle \rightarrow \left\langle \left(\Phi(\bar{u}(s,x))\right)^{*}a; \left(\Phi(\bar{u}(s,x))\right)^{*}b \right\rangle
\end{aligned}
\end{equation}
and
\begin{equation}\nonumber
\begin{aligned}
\mathbb{E} & \bigg\{\bigg[\left\langle  \bar{M}^{ \varepsilon}(t);a \right\rangle \left\langle \bar{M}^{ \varepsilon}(t);b \right\rangle - \left\langle \bar{M}^{ \varepsilon}(s);a \right\rangle \left\langle \bar{M}^{ \varepsilon}(s);b \right\rangle \\
& - \int_{s}^{t} \left\langle \left[\Phi\left(\bar{u}^{ \varepsilon}(s,\xi)\right)\right]^{*}a ; \left[\Phi\left(\bar{u}^{ \varepsilon}(s,\xi)\right)\right]^{*}b \right\rangle \diff \xi\bigg]\varphi(\bar{u}^{ \varepsilon}(t,x))\bigg\} \\
\rightarrow \mathbb{E} & \bigg\{\bigg[\left\langle  \bar{M}(t);a \right\rangle \left\langle \bar{M}(t);b \right\rangle - \left\langle \bar{M}(s);a \right\rangle \left\langle \bar{M}(s);b \right\rangle \\
& - \int_{s}^{t} \left\langle \left[\Phi\left(\bar{u}(s,\xi)\right)\right]^{*}a ; \left[\Phi\left(\bar{u}(s,\xi)\right)\right]^{*}b \right\rangle \diff \xi\bigg]\varphi(\bar{u}(t,x))\bigg\}.
\end{aligned}
\end{equation}

Then $\bar{M}(t)$ is also a square integrable martingale adapted to the filtration  $\sigma\left\{\bar{u}(s), 0\leq s \leq t\right\}$ with quadratic variation equal $\int_{0}^{t} \Phi(\bar{u}(s,x))\left(\Phi(\bar{u}(s,x))\right)^{*} d s$. 

Substitute in the representation theorem (e.g., see Theorem 8.2 in \cite{dPZ}), $M_{t}:=\bar{M}(t)$, $[M_{t}]:=\int_{0}^{t} \Phi(\bar{u}(s,x))\times\left(\Phi(\bar{u}(s,x))\right)^{*} d s$ and $\Phi(s):=\Phi(\bar{u}(s,x))$.

Then there exists a process $\tilde{M}(t) = \int_{0}^{t}\Phi(\bar{u}(s,x))d W(s)$, such that $\tilde{M}(t)=\bar{M}(t)$, $\mathbb{\bar{P}}-\text{a.s.}$, and	
\begin{equation}\nonumber
\begin{aligned}
& \bar{u}(t,x) - u_{0}(x) + \int_{0}^{t}\bigg[\bar{u}(t,x)\bar{u}_{x}(t,x) + \bar{u}_{3x}(t,x) + 3\bar{u}_{x}(t,x) \bar{u}_{2x}(t,x) + \bar{u}(t,x) \bar{u}_{3x}(t,x) \bigg] \diff s \\
&= \int_{0}^{t}\Phi(\bar{u}(s,x))\diff W(s) .
\end{aligned}
\end{equation}
This implies 
\begin{equation}\nonumber
\begin{aligned}
& \bar{u}(t,x) = u_{0}(x) - \int_{0}^{t}\bigg[\bar{u}(t,x)\bar{u}_{x}(t,x) + \bar{u}_{3x}(t,x) + 3\bar{u}_{x}(t,x) \bar{u}_{2x}(t,x) + \bar{u}(t,x) \bar{u}_{3x}(t,x) \bigg]  \diff s \\
& + \int_{0}^{t}\Phi(\bar{u}(s,x))\diff W(s) , \\
\end{aligned}
\end{equation}
so $\bar{u}(t,x)$ is a solution to (\ref{Apr}), what finishes the proof of Theorem \ref{P4.1} .
\end{proof}

%================================
\section{Proofs of Lemmas \ref{szac4.1} and \ref{PropCias}}\label{sec3}
%================================

\begin{proof}[Proof of Lemma \ref{szac4.1}]
Let $p: \mathbb{R} \rightarrow \mathbb{R}$, be a smooth function fulfilling conditions
\begin{itemize}
\item[(i)] $p$ is increasing in $\mathbb{R}$;
\item[(ii)] $\forall_{x\in{\mathbb{R}}}$~ $~p > \delta_{0} > 0$;
\item[(iii)] $\forall_{n\in\mathbb{N}}~ \left|\frac{\partial ^{n}}{\partial x^{n}} p(x)\right|<\delta_{n}$;
\item[(iv)] $\left(\lambda - 2\right) \delta_{2} \geq \delta_{3}$ .
\end{itemize}
Let $F(u^{\varepsilon}) := \int_{X}p(x)(u^{\varepsilon}(x))^{2}\diff x$. Applying the It\^o formula for $F(u^{\varepsilon})$, we obtain
\begin{equation}
\begin{aligned}
\diff F(u^{\varepsilon}(t,x)) = & \left\langle F'(u^{\varepsilon}(t,x));\Phi(u^{\varepsilon}(t,x)) \right\rangle \diff W(t) - \left\langle F'(u^{\varepsilon}(t,x));\varepsilon u^{\varepsilon}_{4x}(t,x) + u^{\varepsilon}_{3x}(t,x) \right.\\
&\left. +  u^{\varepsilon}(t,x)u^{\varepsilon}_{x}(t,x) + 3u_{x}^{\varepsilon}(t,x)u^{\varepsilon}_{2x}(t,x) + u^{\varepsilon}(t,x)u^{\varepsilon}_{3x}(t,x)\right\rangle \diff t \\
&+ \frac{1}{2}tr\left\{F''(u^{\varepsilon}(t,x))\Phi(u^{\varepsilon}(t,x))\left[\Phi(u^{\varepsilon}(t,x))\right]^{*}\right\}\diff t ,
\end{aligned}
\end{equation}
where $$\left\langle F'(u^{\varepsilon}(t,x));v(t,x) \right\rangle  \!= \! 2\!\!\int_{X} \!p(x)u^{\varepsilon}(t,x)v(t,x)\diff x \quad \mbox{and}  \quad
F''(u^{\varepsilon}(t,x))v(t,x)   \!=  \!2p(x)v(t,x). $$

%\begin{lemma}\label{DebL1}[\cite{Deb}, p.242] 
We use the following estimates from \cite[p.242]{Deb}. 
There exist $C_{1}, C_{2}, C_{3}$, such that
\begin{equation}\nonumber
\begin{aligned}
\int_{\mathbb{R}} p(x)u^{\varepsilon}(t,x)u^{\varepsilon}_{4x}(t,x) \diff x \geq & \frac{1}{2}\int_{\mathbb{R}}p(x)\left[u^{\varepsilon}_{2x}(t,x)\right]^{2}\diff x - C_{1}\left|u^{\varepsilon}(t,x)\right|^{2}_{L^{2}(\mathbb{R})} \\
&- C_{2}\int_{\mathbb{R}}p'(x)\left[u_{x}(t,x)\right]^{2}\diff x ; \\
\int_{\mathbb{R}} p(x)u^{\varepsilon}(t,x)u^{\varepsilon}_{3x}(t,x) \diff x \geq & \frac{3}{2}\int_{\mathbb{R}}p'(x)\left[u^{\varepsilon}_{x}(t,x)\right]^{2}\diff x - \frac{1}{2}\int_{\mathbb{R}}p'''(x)\left[u(t,x)\right]^{2}\diff x ; \\
\int_{\mathbb{R}} p(x)\left[u^{\varepsilon}(t,x)\right]^{2}u^{\varepsilon}_{x}(t,x) \diff x = & -\frac{1}{3} \int_{\mathbb{R}}p'(x)\left[u^{\varepsilon}(t,x)\right]^{3} \diff x \\
 \geq & -C_{3}\left(1+\left|u^{\varepsilon}(t,x)\right|_{L^{2}(\mathbb{R})}^{6}\right) - \frac{1}{2}\int_{\mathbb{R}}p'(x)\left[u_{x}(t,x)\right]^{2}\diff x . 
\end{aligned}
\end{equation}
%\end{lemma}
Similarly as above%in Lemma \ref{DebL1}
, one has
\begin{equation} %\nonumber
\begin{aligned}
\int_{\mathbb{R}} & p(x)\left[3u_{x}^{\varepsilon}(t,x)u^{\varepsilon}_{2x}(t,x) + u^{\varepsilon}(t,x)u^{\varepsilon}_{3x}(t,x)\right]
=  \int_{\mathbb{R}}p''(x)u_{x}^{\varepsilon}(t,x)\left[u^{\varepsilon}(t,x)\right]^{2} \diff x  \\
& + \int_{\mathbb{R}}p'(x)\left[u_{x}^{\varepsilon}(t,x)\right]^{2}u^{\varepsilon}(t,x) \diff x + \int_{\mathbb{R}}p(x)u^{\varepsilon}(t,x)u_{x}^{\varepsilon}(t,x)u_{2x}^{\varepsilon}(t,x)  \diff x \\
\geq & - \frac{1}{3} \int_{\mathbb{R}}p'''(x)\left[u^{\varepsilon}(t,x)\right]^{3} \diff x - \int_{\mathbb{R}}p'(x)\left|u^{\varepsilon}(t,x)\right|\left[u_{x}^{\varepsilon}(t,x)\right]^{2} \diff x \\
& - \int_{\mathbb{R}}p(x)\left|u^{\varepsilon}(t,x)\right|u_{x}^{\varepsilon}(t,x)u_{2x}^{\varepsilon}(t,x)  \diff x \\
%\end{aligned}\end{equation}\begin{equation}\begin{aligned}
\geq & -C_{4}\left(1+\left|u^{\varepsilon}(t,x)\right|_{L^{2}(\mathbb{R})}^{6}\right) - \frac{1}{2}\int_{\mathbb{R}}p'''(x)\left[u_{x}(t,x)\right]^{2}\diff x  - \lambda \int_{\mathbb{R}}p'(x)\left[u_{x}^{\varepsilon}(t,x)\right]^{2} \diff x \\
& - \lambda \int_{\mathbb{R}}p(x)u_{x}^{\varepsilon}(t,x)u_{2x}^{\varepsilon}(t,x)  \diff x \\ 
= & -C_{4}\left(1+\left|u^{\varepsilon}(t,x)\right|_{L^{2}(\mathbb{R})}^{6}\right) - \frac{1}{2}\int_{\mathbb{R}}p'''(x)\left[u_{x}(t,x)\right]^{2}\diff x  - \lambda \int_{\mathbb{R}}p'(x)\left[u_{x}^{\varepsilon}(t,x)\right]^{2} \diff x \\
& + \frac{1}{2} \lambda \int_{\mathbb{R}}p'(x)\left(u_{x}^{\varepsilon}(t,x)\right)^{2} \diff x.
\end{aligned}
\end{equation}

In consequence we have
\begin{equation}\label{DebL2}
\begin{aligned}
&\left\langle F'(u^{\varepsilon}(t,x));\varepsilon u^{\varepsilon}_{4x}(t,x) + u^{\varepsilon}_{3x}(t,x) +  u^{\varepsilon}(t,x)u^{\varepsilon}_{x}(t,x) + 3u_{x}^{\varepsilon}(t,x)u^{\varepsilon}_{2x}(t,x) + u^{\varepsilon}(t,x)u^{\varepsilon}_{3x}(t,x)\right\rangle \\
\geq & \varepsilon\int_{\mathbb{R}}p(x)\left[u^{\varepsilon}_{2x}(t,x)\right]^{2}\diff x - 2 \varepsilon C_{1}\int_{\mathbb{R}}\left[u^{\varepsilon}(t,x)\right]^{2} \diff x - 2 \varepsilon C_{2}\int_{\mathbb{R}}p'(x)\left[u_{x}(t,x)\right]^{2}\diff x \\
& + 3\int_{\mathbb{R}}p'(x)\left[u^{\varepsilon}_{x}(t,x)\right]^{2}\diff x - \int_{\mathbb{R}}p'''(x)\left[u(t,x)\right]^{2}\diff x -2 C_{3}\left(1+\left|u^{\varepsilon}(t,x)\right|_{L^{2}(\mathbb{R})}^{6}\right) \\
& - \int_{\mathbb{R}}p'(x)\left[u_{x}(t,x)\right]^{2}\diff x -2 C_{4}\left(1+\left|u^{\varepsilon}(t,x)\right|_{L^{2}(\mathbb{R})}^{6}\right) - \int_{\mathbb{R}}p'''(x)\left[u_{x}(t,x)\right]^{2}\diff x \\
&  - 2 \lambda \int_{\mathbb{R}}p'(x)\left[u_{x}^{\varepsilon}(t,x)\right]^{2} \diff x + \lambda \int_{\mathbb{R}}p'(x)\left[u_{x}^{\varepsilon}(t,x)\right]^{2} \diff x % \\
\end{aligned}
\end{equation}
\begin{equation} \nonumber
\begin{aligned}
= & \varepsilon\int_{\mathbb{R}}p(x)\left[u^{\varepsilon}_{2x}(t,x)\right]^{2}\diff x + \int_{\mathbb{R}} \left[-2\varepsilon C_{2} p'(x) + 3p'(x) - p'(x) - p'''(x) -\lambda p'(x)\right] \left[u_{x}^{\varepsilon}(t,x)\right]^{2} \diff x \\
& + \int_{\mathbb{R}} \left[-2 \varepsilon C_{1} - p'''(x) \right] \left[u^{\varepsilon}(t,x)\right]^{2} \diff x - C_{5} \left(1+\left|u^{\varepsilon}(t,x)\right|_{L^{2}(\mathbb{R})}^{6}\right) \\
= & \varepsilon\int_{\mathbb{R}}p(x)\left[u^{\varepsilon}_{2x}(t,x)\right]^{2}\diff x + \int_{\mathbb{R}} \left[\left(-2\varepsilon C_{2} - \lambda + 2\right) p'(x) - p'''(x) \right] \left[u_{x}^{\varepsilon}(t,x)\right]^{2} \diff x \\
& + \int_{\mathbb{R}} \left[-2 \varepsilon C_{1} - p'''(x) \right] \left[u^{\varepsilon}(t,x)\right]^{2} \diff x - C_{5} \left(1+\left|u^{\varepsilon}(t,x)\right|_{L^{2}(\mathbb{R})}^{6}\right) \\
\geq & \varepsilon\int_{\mathbb{R}}p(x)\left[u^{\varepsilon}_{2x}(t,x)\right]^{2}\diff x + \int_{\mathbb{R}} \left[2\delta_{1}\varepsilon C_{2} +\delta_{1}\left( \lambda - 2\right) - \delta_{2} \right] \left[u_{x}^{\varepsilon}(t,x)\right]^{2} \diff x \\
& + \int_{\mathbb{R}} \left[-2 \varepsilon C_{1} - p'''(x) \right] \left[u^{\varepsilon}(t,x)\right]^{2} \diff x - C_{5} \left(1+\left|u^{\varepsilon}(t,x)\right|_{L^{2}(\mathbb{R})}^{6}\right) \\
\geq & \varepsilon\delta\int_{\mathbb{R}}\left[u^{\varepsilon}_{2x}(t,x)\right]^{2}\diff x + 2\delta_{1}\varepsilon C_{2} \int_{\mathbb{R}} \left[u_{x}^{\varepsilon}(t,x)\right]^{2} \diff x %\\ &
 - \left[2 \varepsilon C_{1} + \delta_{3} \right] \int_{\mathbb{R}} \left[u^{\varepsilon}(t,x)\right]^{2} \diff x - C_{5} \left(1+\lambda^{6}\right).
\end{aligned}
\end{equation}

Let $\left\{e_{1}\right\}_{i\in\mathbb{N}}$ be an orthonormal basis in $L^{2}(\mathbb{R})$. Then there exists a constant $C_{4}>0$, such that
\begin{equation}\label{DebL3}
\begin{aligned}
\text{Tr}\left(F''(u)\Phi(u)\left[\Phi(u)\right]^{*}\right) =& 2\sum_{i\in\mathbb{N}} \int_{\mathbb{R}}p(x)\left|\Phi\left(u^{\varepsilon}(t,x)\right)e_{i}(x)\right|^{2} \diff x \leq C_{4}\left|\Phi\left(u^{\varepsilon}(t,x)\right)\right|^{2}_{L_{0}^{2}\left(L^{2}(X)\right)} \\
\leq & C_{6}\left(\kappa_{1}\left|u^{\varepsilon}(t,x)\right|_{L^{2}(X)}^{2}+\kappa_{2}\right)^{2}.
\end{aligned}
\end{equation}
Due to (\ref{DebL2}) and (\ref{DebL3}) we have
\begin{equation}\nonumber
\begin{aligned}
\mathbb{E}F(u^{\varepsilon}(t,x)) \leq & F\left(u^{\varepsilon}_{0}\right) - \varepsilon\delta\mathbb{E}\int_{0}^{t}\int_{\mathbb{R}}\left[u^{\varepsilon}_{2x}(t,x)\right]^{2}\diff x - 2\delta_{1}\varepsilon C_{2} \mathbb{E}\int_{0}^{t}\int_{\mathbb{R}} \left[u_{x}^{\varepsilon}(t,x)\right]^{2} \diff x \\
& + \left[2 \varepsilon C_{1} + \delta_{3} \right] \mathbb{E}\int_{0}^{t}\int_{\mathbb{R}} \left[u^{\varepsilon}(t,x)\right]^{2} \diff x + tC_{5} \left(1+\lambda^{6}\right) \\
& + C_{6}\mathbb{E}\int_{0}^{t}\left(\kappa_{1}\left|u^{\varepsilon}(t,x)\right|^{2}_{L^{2}(\mathbb{R})}+\kappa_{2}\right)^{2}\diff t, \\
\end{aligned}
\end{equation}
so,
\begin{equation}\nonumber
\begin{aligned}
&\mathbb{E}F(u^{\varepsilon}(t,x)) + \varepsilon\delta\mathbb{E}\int_{0}^{t}\int_{\mathbb{R}}\left[u^{\varepsilon}_{2x}(t,x)\right]^{2}\diff x \diff t + 2\delta_{1}\varepsilon C_{2} \mathbb{E}\int_{0}^{t}\int_{\mathbb{R}} \left[u_{x}^{\varepsilon}(t,x)\right]^{2} \diff x \diff t \\
\leq & F\left(u^{\varepsilon}_{0}\right) + \left[2 \varepsilon C_{1} + \delta_{3} \right] \mathbb{E}\int_{0}^{t}\int_{\mathbb{R}} \left[u^{\varepsilon}(t,x)\right]^{2} \diff x \diff t + C_{5} t\left(1+\lambda^{6}\right) \\
& + C_{6}\mathbb{E}\int_{0}^{t}\left(\kappa_{1}\left|u^{\varepsilon}(t,x)\right|^{2}_{L^{2}(\mathbb{R})}+\kappa_{2}\right)^{2}\diff t \\
= & F\left(u^{\varepsilon}_{0}\right) + \left[2 \varepsilon C_{1} + \delta_{3} \right] t\lambda^{2} + C_{5} t\left(1+\lambda^{6}\right) + C_{6}t\left(\kappa_{1}\lambda^{2}+\kappa_{2}\right)^{2} \\
%%%%%%%%%%%%%%%%%%
\leq & F\left(u^{\varepsilon}_{0}\right) + \left[2 \varepsilon C_{1} + \delta_{3} \right] T\lambda^{2} + C_{5} T\left(1+\lambda^{6}\right) + C_{6}T\left(\kappa_{1}\lambda^{2}+\kappa_{2}\right)^{2} \leq \varepsilon C_{7} + C_{8}.
\end{aligned}
\end{equation}

Let $\varepsilon_{0}>0$ be fixed. Then for all $0<\varepsilon<\varepsilon_{0}$ one has
\begin{equation}\nonumber
\begin{aligned}
\varepsilon\, \mathbb{E} &\left( \left|u^{\varepsilon}(t,x)\right|^{2}_{L^{2}(0,T;H^{2}(X))} \right) =  \varepsilon \,\mathbb{E} \! \int_{0}^{T} \!\! \int_{\mathbb{R}} \! \left[u^{\varepsilon}(t,x)\right] ^{2} \diff x \diff t  + \varepsilon \mathbb{E}  \! \int_{0}^{T} \!\! \int_{\mathbb{R}} \! \left[u_{2x}^{\varepsilon}(t,x)\right] ^{2} \diff x \diff t  \\ &
\leq   \varepsilon\, T \lambda^{2} + \varepsilon\, \mathbb{E}  \! \int_{0}^{T}  \!\!\int_{\mathbb{R}}  \!\left[u_{2x}^{\varepsilon}(t,x)\right] ^{2} \diff x \diff t 
=  \varepsilon\, T \lambda^{2} + \varepsilon\, \mathbb{E}  \! \int_{0}^{T} \!\! \int_{\mathbb{R}} \!\frac{1}{\delta} \delta \left[u_{2x}^{\varepsilon}(t,x)\right] ^{2} \diff x \diff t \\ &
\leq \varepsilon\, T \lambda^{2} + \frac{1}{\delta} \varepsilon\, \delta\, \mathbb{E} \! \int_{0}^{T} \!\! \int_{\mathbb{R}} \! \left[u_{2x}^{\varepsilon}(t,x)\right] ^{2} \diff x 
\leq \varepsilon \,T \lambda^{2} + \frac{1}{\delta}\, \varepsilon\left(\varepsilon C_{7}(T) + C_{8}(T)\right)\\ &
 \leq \varepsilon_{0} T \lambda^{2} + \frac{\varepsilon_{0}^{2} C_{7}(T) + \varepsilon_{0}C_{8}(T)}{\delta},
\end{aligned}
\end{equation}
what proves (\ref{4.1a}). Moreover one has
\begin{equation}\nonumber
\begin{aligned}
\mathbb{E} & \left( \left|u^{\varepsilon}(t,x)\right|^{2}_{L^{2}(0,T;H^{1}(-k,k))} \right) =  \mathbb{E}  \int_{0}^{T} \int_{-k}^{k} \left[u^{\varepsilon}(t,x)\right] ^{2} \diff x \diff t  + \mathbb{E}  \int_{0}^{T} \int_{-k}^{k} \left[u_{x}^{\varepsilon}(t,x)\right] ^{2} \diff x \diff t  \\  &
\leq  \varepsilon\, T \lambda^{2} + \mathbb{E}  \int_{0}^{T} \int_{-k}^{k} \left[u_{x}^{\varepsilon}(t,x)\right] ^{2} \diff x 
\leq \varepsilon\, T \lambda^{2} + \mathbb{E}  \int_{0}^{T} \int_{\mathbb{R}} \left[u_{x}^{\varepsilon}(t,x)\right] ^{2} \diff x \\ &
\leq \varepsilon T \lambda^{2} + \frac{1}{2\delta_{1}\varepsilon C_{2}}  2\delta_{1}\varepsilon C_{2} \mathbb{E} \int_{0}^{T} \int_{\mathbb{R}} \left[u_{x}^{\varepsilon}(t,x)\right] ^{2} \diff x \\ & 
\leq \varepsilon T \lambda^{2} + \frac{1}{2\delta_{1}\varepsilon C_{2}} \left(\varepsilon C_{7}(T) + C_{8}(T)\right) \leq \varepsilon_{0} T \lambda^{2} + \frac{\varepsilon_{0} C_{7}(T) + C_{8}(T)}{2\delta_{1}\varepsilon_{0} C_{2}},
\end{aligned}
\end{equation}
what proves inequality (\ref{4.1c}).
\end{proof}

\begin{proof}[Proof of Lemma \ref{PropCias}]
Let $k\in \mathbb{N}$ be arbitrary fixed and let $0<\varepsilon<\varepsilon_{0}$. Then
\begin{equation}\label{dec}
\begin{aligned}
u^{\varepsilon}(t,x) = & u_{0}^{\varepsilon}(x) -  \int_{0}^{t} \bigg[\varepsilon u^{\varepsilon}_{4x}(t,x) + u^{\varepsilon}_{3x}(t,x) + u^{\varepsilon}(t,x)u^{\varepsilon}_{x}(t,x)  \\
& + 3u^{\varepsilon}_{x}(t,x)u^{\varepsilon}_{2x}(t,x) + u^{\varepsilon}(t,x)u^{\varepsilon}_{3x}(t,x) \bigg] \diff s + \int_{0}^{t} \left( \Phi(u^{\varepsilon}(s,x)) \right) \diff W(s) .
\end{aligned}
\end{equation}

Denote  
\begin{equation}\nonumber
\begin{aligned}
J_{1}:= &  u_{0}^{\varepsilon}(x) ; \quad 
J_{2}:= - \varepsilon\int_{0}^{t} u^{\varepsilon}_{4x}(t,x)  \diff s ; \quad
J_{3}:= - \int_{0}^{t} u^{\varepsilon}(s,x) u^{\varepsilon}_{x}(s,x)  \diff s ; \\
J_{4}:= & - \int_{0}^{t} u^{\varepsilon}_{3x}(t,x) \diff s ; \quad
J_{5}:=  - \left(3\int_{0}^{t}\!\!u^{\varepsilon}_{x}(s,x) u^{\varepsilon}_{2x}(t,x) \diff s + \int_{0}^{t}\!\! u^{\varepsilon}(t,x) u^{\varepsilon}_{3x}(t,x) \diff s\right); \\
J_{6}:= & 	\int_{0}^{t} \left( \Phi(u^{\varepsilon}(s,x)) \right) \diff W(s) .
\end{aligned}
\end{equation}
%\begin{itemize}
%\item[] 
There exists a constant $C_{1}>0$, that ~$\mathbb{E} \left|J_{1}\right|^{2}_{W^{1,2}(0,T,H^{-2}(-k,k))} = C_{1}$.\\
%\item[] 
There exists a constant $C_{2} > 0$, such that
\begin{equation}\nonumber
\left| - \varepsilon u^{\varepsilon}_{4x}(t,x)  \right|_{H^{-2}(-k,k)} =  \varepsilon \left| u^{\varepsilon}_{4x}(t,x)  \right|_{H^{-2}(-k,k)}
 \leq  C_{2}\varepsilon\left| u^{\varepsilon}(s,x) \right|_{H^{2}(-k,k)}.
\end{equation}
Therefore, due to Lemma \ref{szac4.1}, %there exists a constant  $C_{3}(k)>0$, such that
we can write
\begin{equation}\nonumber
\begin{aligned}
\mathbb{E}& \left| - \varepsilon u^{\varepsilon}_{4x}(t,x) \right|^{2}_{L^{2}(0,T;H^{-2}(-k,k))} =  \mathbb{E} \int_{0}^{T} \left| - \varepsilon u^{\varepsilon}_{4x}(t,x) \right|^{2}_{H^{-2}(-k,k)} \diff s \\ &
\leq  C_{2}^{2}\varepsilon^{2} \mathbb{E} \int_{0}^{T} \left| u^{\varepsilon}(s,x) \right|^{2}_{H^{2}(-k,k)} \diff s \leq  C_{3}(k), \mbox{~where~} C_{3}(k)>0.
\end{aligned}
\end{equation}
So, there exists a constant $C_{4}(k)>0$, such that
\begin{equation}\nonumber
\begin{aligned}
\mathbb{E} \left|J_{2}\right|^{2}_{W^{1,2}(0,T,H^{-2}(-k,k))} \leq C_{4}(k) .
\end{aligned}
\end{equation}

%\item[] 
Now, we use the result from \cite[p.243]{Deb}.  
%\begin{lemma}\label{interp}(\cite{Deb}, p.243) 
There exists a constant $C_{5}(k)>0$,  that the following inequality holds
\begin{equation}
\left|u^{\varepsilon}(s,x)u^{\varepsilon}_{x}(s,x)\right|_{H^{-1}(-k,k)} \leq C_{5}(k)\left|u^{\varepsilon}(s,x)\right|^{\frac{3}{2}}_{L^{2}(-k,k)}\left|u^{\varepsilon}(s,x)\right|^{\frac{1}{2}}_{H^{1}(-k,k)}.
\end{equation}
%\end{lemma}
This estimate implies the existence of a constant  $C_{8}(k)>0$, such that
\begin{equation}\nonumber
\begin{aligned}
& \left| - u^{\varepsilon}(s,x)u^{\varepsilon}_{x}(s,x) \right|_{H^{-2}(-k,k)} = \left| u^{\varepsilon}(s,x)u^{\varepsilon}_{x}(s,x)  \right|_{H^{-2}(-k,k)} \\ &
\leq C_{6} \left| u^{\varepsilon}(s,x)u^{\varepsilon}_{x}(s,x) \right|_{H^{-1}(-k,k)}
\leq C_{7}(k) \left|u^{\varepsilon}(s,x)\right|^{\frac{3}{2}}_{L^{2}(-k,k)}\left|u^{\varepsilon}(s,x)\right|^{\frac{1}{2}}_{H^{1}(-k,k)}   \\ &
\leq C_{7}(k) \left|u^{\varepsilon}(s,x)\right|_{L^{2}(-k,k)}\left|u^{\varepsilon}(s,x)\right|^{\frac{1}{2}}_{L^{2}(-k,k)}\left|u^{\varepsilon}(s,x)\right|^{\frac{1}{2}}_{H^{1}(-k,k)} \\ &
\leq C_{7}(k) \left[\left(2k\lambda^{2}\right)^{\frac{1}{2}}\right]\left|u^{\varepsilon}(s,x)\right|^{\frac{1}{2}}_{H^{1}(-k,k)} 
\leq C_{8}(k) \lambda\left|u^{\varepsilon}(s,x)\right|_{H^{1}(-k,k)}  .
\end{aligned}
\end{equation}

Due to Lemma \ref{szac4.1} there exists a constant  $C_{9}(k)>0$, that we can write 
\begin{equation}\nonumber
\begin{aligned}
\mathbb{E}& \left|- u^{\varepsilon}(s,x)u^{\varepsilon}_{x}(s,x) \right|^{2}_{L^{2}(0,T;H^{-2}(-k,k))} =  \mathbb{E} \int_{0}^{T} \left| - u^{\varepsilon}(s,x)u^{\varepsilon}_{x}(s,x) \right|^{2}_{H^{-2}(-k,k)} \diff s \\ & 
\leq C_{8}^{2}(k) \lambda^{2} \mathbb{E} \int_{0}^{T} \left|u^{\varepsilon}(s,x)\right|^{2}_{H^{1}(-k,k)} \diff s 
=  C_{8}^{2}(k) \lambda^{2} \mathbb{E} \left|u^{\varepsilon}(s,x)\right|^{2}_{L^{2}(0,T;H^{1}(-k,k))} \leq  C_{9}(k) \lambda^{2} .
\end{aligned}
\end{equation}

Then, there exists a constant  $C_{10}(k)>0$, such that
%\begin{equation}\nonumber\begin{aligned}
$$\mathbb{E} \left|J_{3}\right|^{2}_{W^{1,2}(0,T,H^{-2}(-k,k))} \leq C_{10}(k).$$
%\end{aligned}\end{equation}

%\item[] There exists a constant  $C_{12} > 0$, such that
We have 
\begin{equation}\nonumber
\begin{aligned}
\left| - u^{\varepsilon}_{3x}(t,x)\right|_{H^{-2}(-k,k)} = & \left| u^{\varepsilon}_{3x}(t,x)  \right|_{H^{-2}(-k,k)}
\leq  C_{11}\left| u^{\varepsilon}(s,x) \right|_{H^{1}(-k,k)} \\
\leq & C_{12}\left| u^{\varepsilon}(s,x) \right|_{H^{2}(-k,k)}, \mbox{~where~} C_{12}>0.
\end{aligned}
\end{equation}
Lemma \ref{szac4.1} implies the existence of a constant $C_{13}>0$, such that
\begin{equation}\nonumber
\begin{aligned}
\mathbb{E} & \left|- u^{\varepsilon}_{3x}(t,x) \right|^{2}_{L^{2}(0,T;H^{-2}(-k,k))} = \mathbb{E} \! \int_{0}^{T}\!\!\! \left| - u^{\varepsilon}_{3x}(t,x) \right|^{2}_{H^{-2}(-k,k)} \diff s  \leq  
C^{2}_{12} \,\mathbb{E} \!\int_{0}^{T}\!\!\! \left|u^{\varepsilon}(s,x)\right|^{2}_{H^{2}(-k,k)} \diff s \\ &
= C^{2}_{12} \,\mathbb{E} \left|u^{\varepsilon}(s,x)\right|^{2}_{L^{2}(0,T;H^{2}(-k,k))} \leq %& 
C^{2}_{12} \,\mathbb{E} \left|u^{\varepsilon}(s,x)\right|^{2}_{L^{2}(0,T;H^{2}(\mathbb{R}))} \leq  C_{13} .
\end{aligned}
\end{equation}
So, there exists a constant  $C_{14}>0$, such that
%\begin{equation}\nonumber\begin{aligned}
$~\mathbb{E} \left|J_{4}\right|^{2}_{W^{1,2}(0,T,H^{-2}(-k,k))} \leq C_{14}$.
%\end{aligned}\end{equation}

%\item[] 
There exist constants $C_{15}, C_{16}(k) > 0$, such that
\begin{equation}\nonumber
\begin{aligned}
\left| \right. - &  \left.\left(3 u^{\varepsilon}_{x}(s,x) u^{\varepsilon}_{2x}(t,x) + u^{\varepsilon}(t,x) u^{\varepsilon}_{3x}(t,x) \right)\right|_{H^{-2}(-k,k)}\\\leq & C_{15}\left| u^{\varepsilon}(s,x)u^{\varepsilon}_{x}(s,x) \right|_{L^{2}(-k,k)} 
\leq  C_{16}(k) \lambda^{2}\left|u^{\varepsilon}(s,x)\right|_{H^{1}(-k,k)}.
\end{aligned}
\end{equation}
Due to Lemma \ref{szac4.1} there exists a constant  $C_{17}(k)>0$, such that
\begin{equation}\nonumber
\begin{aligned}
\mathbb{E} & \left| - \left(3 u^{\varepsilon}_{x}(s,x) u^{\varepsilon}_{2x}(t,x) + u^{\varepsilon}(t,x) u^{\varepsilon}_{3x}(t,x) \right) \right|^{2}_{L^{2}(0,T;H^{-3}(-k,k))} \\ &
 =  \mathbb{E} \int_{0}^{T} \left| - \left(3 u^{\varepsilon}_{x}(s,x) u^{\varepsilon}_{2x}(t,x) + u^{\varepsilon}(t,x) u^{\varepsilon}_{3x}(t,x) \right) \right|^{2}_{H^{-3}(-k,k)} \diff s \\ &
\leq C_{16}^{2}(k) \lambda^{4} \mathbb{E} \! \int_{0}^{T}\!\!\! \left|u^{\varepsilon}(s,x)\right|^{2}_{H^{1}(-k,k)} \diff s 
=  C_{16}^{2}(k) \lambda^{4} \mathbb{E} \left|u^{\varepsilon}(s,x)\right|^{2}_{L^{2}(0,T;H^{1}(-k,k))}  \leq C_{17}(k) \lambda^{4} .
\end{aligned}
\end{equation}
So, there exists a constant  $C_{18}(k)>0$, such that 
%\begin{equation}\nonumber\begin{aligned}
$~\mathbb{E} \left|J_{5}\right|^{2}_{W^{1,2}(0,T,H^{-3}(-k,k))} \leq C_{18}(k)$.
%\end{aligned}\end{equation}

%\item[] 
Substitute in  \cite[Lemma 2.1]{Gat} $f(s) := \Phi(u(s,x))$, $K=H=L^{2}(\mathbb{R})$. Then %\linebreak 
$\mathscr{I}(f)(t) = \int_{0}^{t}\Phi(u(s,x)) \diff W(s)$ and for all $p\geq 1$ and $\alpha<\frac{1}{2}$ there exists a constant  $C_{22}(p,\alpha)>0$, such that
\begin{equation}\nonumber
\begin{aligned}
\mathbb{E}\left|\int_{0}^{t}\Phi(u^{m}(s,x)) \diff W(s)\right|^{2p}_{W^{\alpha ( p),2p}(0,T;L^{2}(\mathbb{R}))} \leq & C_{22}(2p,\alpha) \mathbb{E} \left( \int_{0}^{T} \left|\Phi(u^{m}(s,x))\right|^{2p}_{L_{2}^{0}(L^{2}(\mathbb{R}))} \diff s \right). 
\end{aligned}
\end{equation}
Then, due to condition (\ref{W1}), there exists a constant $C_{23}>0$, that
\begin{equation}\nonumber
\begin{aligned}
\mathbb{E}\left|\int_{0}^{t}\Phi(u^{m}(s,x)) \diff W(s)\right|^{2p}_{W^{\alpha,2p}(0,T;L^{2}(\mathbb{R}))} \leq & C_{23} (p,\alpha) .
\end{aligned}
\end{equation}
Substitution in the above inequality $p:=1$ yields
\begin{equation}\label{It\^o3}
\mathbb{E}\left|J_{6}\right|^{2}_{W^{\alpha,2}(0,T;L^{2}(\mathbb{R}))} = \mathbb{E}\left|\int_{0}^{t}\Phi(u(s,x)) \diff W(s)\right|^{2}_{W^{\alpha,2}(0,T;L^{2}(\mathbb{R}))} \leq C_{23}(2,\alpha) = C_{24}(\alpha). 
\end{equation}
%\end{itemize}

Let $\beta\in\left(0,\frac{1}{2}\right)$ and $\alpha\in\left(\beta + \frac{1}{2}, \infty\right)$ be arbitrary fixed. Note, that the following inclusion relations hold
$$\hspace{2.5ex} W^{\alpha,2}(0,T;L^{2}(\mathbb{R})) \subset W^{\alpha,2}(0,T;H^{-2}([-k,k));$$
$$\mbox{and} \hspace{3ex} W^{1,2}(0,T,H^{-2}(-k,k)) \subset W^{\alpha,2}(0,T,H^{-2}(-k,k)).$$
Then, there exists a constant $C_{25}(\alpha) > 0$, such that
\begin{equation}\nonumber
\begin{aligned}
\mathbb{E} & \left|u^{m}(s,x)\right|_{W^{\alpha,2}(0,T,H^{-2}(-k,k))}^{2}  =  \mathbb{E}\left|\sum_{i=1}^{6} J_{i}\right|_{W^{\alpha,2}(0,T,H^{-2}(-k,k))}^{2} \leq \mathbb{E} \left( \sum_{i=1}^{6} \left|J_{i}\right|_{W^{\alpha,2}(0,T,H^{-2}(-k,k))} \right)^{2} \\
= & \mathbb{E} \left[ \sum_{i=1}^{6} \left|J_{i}\right|^{2}_{W^{\alpha,2}(0,T,H^{-2}(-k,k))} + 2\sum_{i=1}^{6} \sum_{j=i+1}^{6} \left|J_{i}\right|_{W^{\alpha,2}(0,T,H^{-2}(-k,k))}\left|J_{j}\right|_{W^{\alpha,2}(0,T,H^{-2}(-k,k))} \right]\\
\leq & \mathbb{E} \left[ \sum_{i=1}^{6} \left|J_{i}\right|^{2}_{W^{\alpha,2}(0,T,H^{-2}(-k,k))} + 2\sum_{i=1}^{6} \sum_{j=i+1}^{6} \left(\left|J_{i}\right|^{2}_{W^{\alpha,2}(0,T,H^{-2}(-k,k))} + \left|J_{j}\right|^{2}_{W^{\alpha,2}(0,T,H^{-2}(-k,k))}\right) \right]\\
= & \mathbb{E} \left[ 8 \sum_{i=1}^{6} \left|J_{i}\right|^{2}_{W^{\alpha,2}(0,T,H^{-2}(-k,k))}\right]
= 8 \sum_{i=1}^{6} \left[ \mathbb{E} \left|J_{i}\right|^{2}_{W^{\alpha,2}(0,T,H^{-2}(-k,k))} \right]
\leq C_{25}(\alpha) .
\end{aligned}
\end{equation}
Moreover
$$\hspace{7ex} W^{\alpha,2}(0,T,H^{-2}(-k,k)) \subset C^{\beta}(0,T;H^{-3}_{loc}(-k,k);$$
$$\mbox{and} \hspace{3ex} W^{\alpha,2}(0,T,H^{-2}(\mathbb{R})) \subset W^{\alpha,2}(0,T,H^{-2}(-k,k)).$$

So, there exist constants $C_{27}(k), C_{28}(k, \alpha) >0$, such that
\begin{equation}\label{Cszac}
\begin{aligned}
& \mathbb{E}\left|u^{\varepsilon}(s,x)\right|_{C^{\beta}(0,T;H^{-3}(-k,k)}^{2} \leq  C_{26} \mathbb{E}\left|u^{\varepsilon}(s,x)\right|_{W^{\alpha,2}(0,T,H^{-3}(-k,k))}^{2} \leq C_{27}(k,\alpha) \\
& \mathbb{E}\left|u^{\varepsilon}(s,x)\right|_{W^{\alpha,2}(0,T,H^{-2}(-k,k))} \leq  C_{28}(k, \alpha).
\end{aligned}
\end{equation}

Let $\eta>0$ be arbitrary fixed. Due to Lemma \ref{szac4.1} there exists a constant  $C_{30}(k)>0$, that
\begin{equation}\label{rwnszac}
\begin{aligned}
\mathbb{E}\left|u^{\varepsilon}(s,x)\right|^{2}_{L^{2}(0,T,H^{-1}(-k,k))} \leq & C_{29}(k)\mathbb{E}\left|u^{\varepsilon}(s,x)\right|^{2}_{L^{2}(0,T,H^{-1}(\mathbb{R}))} \tilde{C}_{2} = C_{30}(k).
\end{aligned}
\end{equation}

Substituting in \cite[Lemma 2.1]{Deb} $\alpha_{k}:=\eta^{-1}2^{k} \left( C_{30}(k) + C_{27}(k,\alpha) + C_{28}(k,\alpha) \right)$ and using Markov inequality \cite[p. 114]{Pap} for $$X := \left|u^{\varepsilon}(s,x)\right|^{2}_{L^{2}(0,T,H^{-1}(-k,k))} + \left|u^{\varepsilon}(s,x)\right|^{2}_{W^{\alpha,2}(0,T,H^{-2}(-k,k))} + \left|u^{\varepsilon}(s,x)\right|_{C^{\beta}(0,T;H^{-3}_{loc}(-k,k)}^{2}$$ and $ \varepsilon := \eta^{-1}2^{k} \left( C_{30}(k) + C_{27}(k,\alpha) + C_{28}(k,\alpha) \right)$, we obtain
\begin{equation}\nonumber
\begin{aligned}
&\mathbb{P} \Big(u^{\varepsilon} \in A\left(\left\{\alpha _{k} \right\} \right) \Big) = 
  1 - \mathbb{P} \Big( \left|u^{\varepsilon}(s,x)\right|^{2}_{L^{2}(0,T,H^{-1}(-k,k))} + \left|u^{\varepsilon}(s,x)\right|^{2}_{W^{\alpha,2}(0,T,H^{-2}(-k,k))}   \\ 
&  ~+ \left|u^{\varepsilon}(s,x)\right|_{C^{\beta}(0,T;H^{-3}_{loc}(-k,k))}^{2}  \geq \eta^{-1}2^{k} \left( C_{30}(k) + C_{27}(k,\alpha)  + C_{28}(k,\alpha) \right) \Big) \\
= &\hspace{2ex} 1 - \frac{C_{30}(k) + C_{27}(k,\alpha) + C_{28}(k,\alpha)}{\eta^{-1}2^{k} \left( C_{30}(k) + C_{27}(k,\alpha) + C_{28}(k,\alpha)\right)} 
=  1 - \frac{\eta}{2^{k}} > 1 - \eta .
\end{aligned}
\end{equation}

Let $K$ be the following mapping for  $\eta>0$:   $K\left( \eta \right) = A\left(\left\{a_{k}^{(\eta)}\right\}\right)$, where $\left\{a_{k}^{(\eta)}\right\}$ is an increasing sequence of positive numbers, which can, but does not have to, depend on  $\eta$. Note, that due to \cite[Lemma 2.1]{Deb}, the set $K(\eta)$ is compact for all  $\eta>0$. Moreover, $\mathbb{P}\left\{K\left( \eta \right)\right\} > 1-\eta$, then the family $\mathscr{L}\left(u^{\varepsilon}\right)$ is tight. 
\end{proof}

\section{Proof of Lemma \ref{parMart}} \label{sec4}
\begin{proof} %[Dowód Lematu \ref{parMart}]
Let $\left\{e_{i}\right\}_{i\in\mathbb{N}}$ be an orthonormal basis in space  $L^{2}(\mathbb{R})$. Denote by $P_{m}$, for all $m\in\mathbb{N}$, the orthogonal projection on $Sp(e_{0},...,e_{m})$. Consider finite dimensional approximation of the problem (\ref{par}) in the space  $P_{m}L^{2}(\mathbb{R})$ of the form
\begin{equation}\label{Galerkin}
\begin{cases}
\diff u^{m,\varepsilon}(t,x) + \left[\varepsilon \theta\left(\frac{\left|u^{m,\varepsilon}_{4x}(t,x)\right|^{2}}{m}\right) u^{m,\varepsilon}_{4x}(t,x) + \theta\left(\frac{\left|u^{m,\varepsilon}_{x}(t,x)\right|^{2}}{m} \right) u^{m,\varepsilon}(t,x)u^{m,\varepsilon}_{x}(t,x)\right. \\ \hspace{14ex}
+ \theta\left(\frac{\left|u^{m,\varepsilon}_{3x}(t,x)\right|^{2}}{m} \right)u^{m,\varepsilon}_{3x}(t,x) + 3 \theta\left(\frac{\left|u^{m,\varepsilon}_{x}(t,x)u^{m,\varepsilon}_{2x}(t,x)\right|^{2}}{m} \right)u^{m,\varepsilon}_{2x}(t,x) \\ \left. \hspace{14ex}
+ \theta\left(\frac{\left|u^{m,\varepsilon}_{3x}(t,x)\right|^{2}}{m} \right)u^{m,\varepsilon}(t,x)u^{m,\varepsilon}_{3x}(t,x)  \right] \diff t
=  P_{m}\Phi\left(u^{m,\varepsilon}(t,x)\right)\diff W^{m}(t) \\
u^{m,\varepsilon}_{0}(x) = P_{m}u^{\varepsilon}(0,x) ,
\end{cases}
\end{equation}
where $\theta\in C^{\infty}(\mathbb{R})$ fulfils conditions
\begin{equation}
\begin{cases}
\theta(\xi) = 1, \quad &\textrm{when} \quad \xi\in [0,1] \\
\theta(\xi) \in [0,1], \quad &\textrm{when} \quad \xi\in (1,2) \\
\theta(\xi) = 0, \quad &\textrm{when} \quad \xi\in \left.[2,\infty)\right. .
\end{cases}
\end{equation}

Let $m\in\mathbb{N}$ be arbitrary fixed and 
\begin{equation}\nonumber
\begin{aligned}
b(u(t,x)) := & \theta\left(\frac{\left|u^{m,\varepsilon}_{x}(t,x)\right|^{2}}{m} \right) u^{m,\varepsilon}(t,x)u^{m,\varepsilon}_{x}(t,x) + \theta\left(\frac{\left|u^{m,\varepsilon}_{3x}(t,x)\right|^{2}}{m} \right) u^{m,\varepsilon}(t,x)u^{m,\varepsilon}_{3x}(t,x)   \\
&  +  3 \theta\left(\frac{\left|u^{m,\varepsilon}_{x}(t,x)u^{m,\varepsilon}_{2x}(t,x)\right|^{2}}{m} \right)u^{m,\varepsilon}_{x}(t,x)u^{m,\varepsilon}_{2x}(t,x)  , \\
\sigma (u(t,x)) := & \Phi(u^{m,\varepsilon}(t,x)).
\end{aligned}
\end{equation} 
Then
\begin{equation}\nonumber
\begin{aligned}
\left| b(u(t,x)) \right|_{L^{2}(\mathbb{R})} \leq & \left| \theta\left(\frac{\left|u^{m,\varepsilon}_{x}(t,x)\right|^{2}}{m} \right) u^{m,\varepsilon}(t,x)u^{m,\varepsilon}_{x}(t,x) \right|_{L^{2}(\mathbb{R})} \\
& + \left| \theta\left(\frac{\left|u^{m,\varepsilon}_{3x}(t,x)\right|^{2}}{m} \right) u^{m,\varepsilon}(t,x)u^{m,\varepsilon}_{3x}(t,x) \right|_{L^{2}(\mathbb{R})} \\
& + 3 \left| \theta\left(\frac{\left|u^{m,\varepsilon}_{x}(t,x)u^{m,\varepsilon}_{2x}(t,x)\right|_{L^{2}(\mathbb{R})}^{2}}{m} \right)u^{m,\varepsilon}_{x}(t,x)u^{m,\varepsilon}_{2x}(t,x) \right|_{L^{2}(\mathbb{R})} \\
= : & J_{1} + J_{2} + 3 J_{3}.
\end{aligned}
\end{equation}
Note, that
\begin{equation}\nonumber
J_{1} = 
\begin{cases}
0, \quad \mbox{when} \quad  \left|u^{m,\varepsilon}_{x}(t,x)\right| \geq \sqrt{2m} \\
\lambda  \left| u^{m,\varepsilon}(t,x)u^{m,\varepsilon}_{x}(t,x) \right|_{L^{2}(\mathbb{R})}, \quad \mbox{when} \quad \left|u^{m,\varepsilon}_{x}(t,x)\right| \leq \sqrt{2m} 
\end{cases}
\end{equation}
where $\lambda \in [0,1]$, therefore
\begin{equation}\nonumber
\begin{aligned}
J_{1} \leq \left| u^{m,\varepsilon}(t,x)u^{m,\varepsilon}_{x}(t,x) \right|_{L^{2}(\mathbb{R})} \leq \sqrt{2m} \left| u^{m,\varepsilon}(t,x) \right|_{L^{2}(\mathbb{R})} .
\end{aligned}
\end{equation}
Analogously, 
\begin{equation}\nonumber
\begin{aligned}
J_{2} \leq \left| u^{m,\varepsilon}(t,x)u^{m,\varepsilon}_{3x}(t,x) \right|_{L^{2}(\mathbb{R})} \leq \sqrt{2m} \left| u^{m,\varepsilon}(t,x) \right|_{L^{2}(\mathbb{R})} .
\end{aligned}
\end{equation}
Moreover
\begin{equation}\nonumber
J_{3} = 
\begin{cases}
0, \quad \mbox{when} \quad  \left|u^{m,\varepsilon}_{x}(t,x)u^{m,\varepsilon}_{2x}(t,x)\right|_{L^{2}(\mathbb{R})}^{2} \geq \sqrt{2m} \\
\lambda  \left|u^{m,\varepsilon}_{x}(t,x)u^{m,\varepsilon}_{2x}(t,x)\right|^{2}, \quad \mbox{when} \quad \left|u^{m,\varepsilon}_{x}(t,x)u^{m,\varepsilon}_{2x}(t,x)\right|_{L^{2}(\mathbb{R})}^{2} \leq \sqrt{2m} 
\end{cases}
\end{equation}
where $\lambda \in [0,1]$, so
\begin{equation}\nonumber
\begin{aligned}
3J_{3} \leq 3 \left| u^{m,\varepsilon}(t,x)u^{m,\varepsilon}_{x}(t,x) \right|_{L^{2}(\mathbb{R})} \leq 3 \sqrt{2m} .
\end{aligned}
\end{equation}
Finally, 
\begin{equation}\nonumber
\begin{aligned}
& \left| b( u^{m,\varepsilon}(t,x)) \right|_{L^{2}(\mathbb{R})} \leq 2\sqrt{2m} \left| u^{m,\varepsilon}(t,x) \right|_{L^{2}(\mathbb{R})} + 3 \sqrt{2m} .
\end{aligned}
\end{equation}
Additionally, due to the condition (\ref{W1}), there exist constants $\kappa_{1}, \kappa_{2} > 0$, such that
\begin{equation}
\left\|\Phi(u^{m,\varepsilon}(t,x))\right\|_{L_{0}^{2}(L^{2}(\mathbb{R}))} \leq \kappa_{1} \left|u^{m,\varepsilon}(t,x)\right|_{L^{2}(\mathbb{R})} + \kappa_{2} ,
\end{equation}
then 
\begin{equation}\nonumber
\begin{aligned}
& \left| b(u^{m,\varepsilon}(t,x)) \right|_{L^{2}(\mathbb{R})} + \left\|\sigma(u^{m,\varepsilon}(t,x))\right\|_{L_{0}^{2}(L^{2}(\mathbb{R}))} \\
\leq & 2\sqrt{2m} \left| u^{m,\varepsilon}(t,x) \right|_{L^{2}(\mathbb{R})} + 3 \sqrt{2m} + \kappa_{1} \left|u^{m,\varepsilon}(t,x)\right|_{L^{2}(\mathbb{R})}+ \kappa_{2} \\
= & \left( 2\sqrt{2m} + \kappa_{1} \right) \left| u^{m,\varepsilon}(t,x) \right|_{L^{2}(\mathbb{R})} + 3 \sqrt{2m} + \kappa_{2} \\
\leq & \left( 3\sqrt{2m} + \max \left\{ \kappa_{1}, \kappa_{2} \right\} \right) \left| u^{m,\varepsilon}(t,x) \right|_{L^{2}(\mathbb{R})} +  3\sqrt{2m} + \max \left\{ \kappa_{1}, \kappa_{2} \right\} \\
= & \left( 3\sqrt{2m} + \max \left\{ \kappa_{1}, \kappa_{2} \right\} \right) \left( \left| u^{m,\varepsilon}(t,x) \right|_{L^{2}(\mathbb{R})} + 1\right).
\end{aligned}
\end{equation}

Therefore, from \cite[Prop.~3.6 and 4.6]{Kar}, when $b(u(t,x))$ and $\sigma(u(t,x))$ are as above, for all $m\in\mathbb{N}$, there exists a martingale solution to (\ref{Galerkin}). Moreover, applying the same methods as in section \ref{sec3} one can show that
%\begin{itemize} \item[(i)] 
for all $m$ the following inequalities hold
\begin{eqnarray}
\exists_{C_{1}(\varepsilon)>0}\mathbb{E}\left( \left|u^{m,\varepsilon}(t,x)\right|^{2}_{L^{2}(0,T;H^{2}(\mathbb{R}))} \right)  & \leq \tilde{C}_{1}(\varepsilon),  \\ 
\forall_{k\in X_{k}}\exists_{C_{2}(k,\varepsilon)>0} \mathbb{E}\left( \left|u^{m,\varepsilon}(t,x)\right|^{2}_{L^{2}(0,T;H^{1}(-k,k))} \right)  & \leq \tilde{C}_{2}(k,\varepsilon);
\end{eqnarray}
%\item[(ii)] 
and the family of distributions $\mathscr{L}(u^{m,\varepsilon})$ is tight in $L^{2}(0,T;L^{2}_{loc})\cap C(0,T;H^{-3}_{loc}(\mathbb{R}))$.
%\end{itemize}
Then application of the same methods, as used already on pages \pageref{par}--\pageref{sec3}, leads to the  proof of the existence of martingale solution to  (\ref{par}).
\end{proof}

%\newpage

\end{document}